\documentclass[12pt]{amsart}

\usepackage{ucs}

\usepackage{amssymb}
\usepackage{amsthm}
\usepackage{amsmath}
\usepackage{latexsym}
\usepackage[cp1251]{inputenc}
\usepackage{graphicx}
\usepackage{wrapfig}
\usepackage{caption}
\usepackage{subcaption}
\usepackage{indentfirst}
\usepackage[left=2.5cm,right=2.5cm,top=2.4cm,bottom=2.4cm,bindingoffset=0cm]{geometry}
\usepackage{enumerate}
\usepackage{makecell}

\DeclareMathOperator{\aut}{Aut}

\DeclareMathOperator{\cay}{Cay}
\DeclareMathOperator{\cyc}{Cyc}

\DeclareMathOperator{\rk}{rk}

\DeclareMathOperator{\Span}{Span}

\DeclareMathOperator{\sym}{Sym}
\DeclareMathOperator{\rad}{rad}

\DeclareMathOperator{\dimwl}{dim_{WL}}
\DeclareMathOperator{\rkwl}{rk_{WL}}
\DeclareMathOperator{\WL}{WL}
\DeclareMathOperator{\Paley}{Paley}
\DeclareMathOperator{\GCD}{GCD}
\DeclareMathOperator{\GPaley}{GPaley}

\makeatletter 
\def\@seccntformat#1{\csname the#1\endcsname. } 
\def\@biblabel#1{#1.}

\makeatother

\title{On WL-rank and WL-dimension of some Deza circulant graphs}

\author{Ravil Bildanov}
\address{Novosibirsk State University, Novosibirsk, Russia}
\email{ravilbildanov@gmail.com}

\author{Viktor Panshin}
\address{Novosibirsk State University, Novosibirsk, Russia}
\email{v.panshin@g.nsu.ru}

\author{Grigory Ryabov}
\address{Sobolev Institute of Mathematics, Novosibirsk, Russia}
\address{Novosibirsk State University, Novosibirsk, Russia}
\email{gric2ryabov@gmail.com}

\thanks{The work is supported by Mathematical Center in Akademgorodok under agreement No.~075-2019-1675 with the Ministry of Science and Higher Education of the Russian Federation}

\date{}

\newtheorem{state}{Statement}[section]

\newtheorem{lemm}[state]{Lemma}
\newtheorem{theo}[state]{Theorem}
\newtheorem{prop}[state]{Proposition}

\theoremstyle{definition}

\newtheorem*{rem1}{Remark 1}

\begin{document}

\vspace{\baselineskip}
\vspace{\baselineskip}

\vspace{\baselineskip}

\vspace{\baselineskip}

\begin{abstract}
The \emph{WL-rank} of a digraph $\Gamma$ is defined to be the rank of the coherent configuration of $\Gamma$. The \emph{WL-dimension} of $\Gamma$ is defined to be the smallest positive integer~$m$ for which $\Gamma$ is identified by the $m$-dimensional Weisfeiler-Leman algorithm. We classify the Deza circulant graphs of WL-rank~$4$. In additional, it is proved that each of these graphs has WL-dimension at most~$3$. Finally, we establish that some families of Deza circulant graphs have WL-rank~$5$ or~$6$ and WL-dimension at most~$3$.
\end{abstract}

\maketitle

\section{Introduction}

Let $V$ be a finite set and $|V|=n$. A \emph{coherent configuration} on $V$ can be thought as a special partition of $V\times V$ for which the diagonal of $V\times V$ is a union of some classes~\cite[Definition~2.1.3]{CP}. Let $\Gamma=(V,E)$ be a digraph with vertex set $V$ and arc set $E$. The \emph{WL-rank} (the \emph{Weisfeiler-Leman rank}) of $\Gamma$ is defined to be the number of classes in the smallest coherent configuration on the set $V$ for which $E$ is a union of classes (see Section~$2$ for the exact definitions). For a given graph, such coherent configuration can be found efficiently using the Weisfeiler-Leman algorithm~\cite{WeisL}. This explains the choice of the term ``WL-rank''. The WL-rank of $\Gamma$ is denoted by $\rkwl(\Gamma)$. Since the diagonal of $V\times V$ is a union of some classes, $\rkwl(\Gamma)\geq 2$. It is easy to check that $\rkwl(\Gamma)=2$ if and only if $\Gamma$ is complete or empty. On the other hand, obviously, $\rkwl(\Gamma)\leq n^2$.

A $k$-regular graph $\Gamma$ is called \emph{strongly regular} if there exist nonnegative integers $\lambda$ and $\mu$ such that every two adjacent vertices have $\lambda$ common neighbors and every two distinct non-adjacent vertices have $\mu$ common neighbors. The following generalization of strongly regular graphs was introduced in~\cite{Deza} and the name was given in~\cite{EFHHH}. A $k$-regular graph $\Gamma$ with $n$ vertices is called a \emph{Deza} graph if there exist nonnegative integers $a$ and $b$ such that any pair of distinct vertices of $\Gamma$ has either $a$ or $b$ common neighbors. The numbers $(n,k,b,a)$ are called the \emph{parameters} of the Deza graph $\Gamma$. Unlike the class of strongly regular graphs, the class of Deza graphs is not closed under taking complements~\cite{EFHHH}. A Deza graph is called \emph{strictly} if it is non-strongly regular and has diameter~$2$. Clearly, if $a>0$ and $b>0$ then $\Gamma$ has diameter~$2$. For more information on Deza graphs, we refer the readers to~\cite{EFHHH}. The WL-rank of a given graph is at most~$3$ if and only if this graph is strongly regular~\cite[Section~2.6.3]{CP}. It seems natural to study how large the WL-rank of a (strictly) Deza graph can be.

Let $G$ be a finite group. By a \emph{Cayley digraph} over $G$, we mean a digraph with vertex set $G$ and arc set $\{(g,sg):~s\in S,g\in G\}$, where $S$ is an identity-free subset of $G$. For a given identity-free subset $S$ of $G$, such digraph is denoted by $\cay(G,S)$. If $S=S^{-1}$ then $\Gamma$ is a \emph{Cayley graph}.  In the present paper we deal with Deza \emph{circulant} graphs, i.e. Deza Cayley graphs over cyclic groups. All strongly regular circulant graphs, i.e. all circulant graphs of WL-rank at most~$3$, were classified independently in~\cite{BM,HLW,Ma}. We classify all Deza circulant graphs of WL-rank~$4$. 

In what follows, the cyclic group of order $n$ is denoted by $C_n$. The complete and cycle graphs with $n$ vertices are denoted by $K_n$ and $\mathcal{C}_n$ respectively. The complete bipartite graph with parts of sizes $l$ and $m$ is denoted by $K_{l,m}$. The complement to a graph $\Gamma$ is denoted by $\overline{\Gamma}$. The tensor and the lexicographic products of graphs $\Gamma_1$ and $\Gamma_2$ are denoted by $\Gamma_1\times \Gamma_2$ and $\Gamma_1[\Gamma_2]$ respectively (see~\cite{IK} for the definitions). The disjoint union of $m$ copies of a graph $\Gamma$ is denoted by $m\Gamma$. If $p$ is a prime and $k$ is a divisor of $p-1$ such that $\frac{p-1}{k}$ is even then the generalized Paley graph (see~\cite[Definition~1.1]{LP}) with $p$ vertices of valency $\frac{p-1}{k}$ is denoted by $\GPaley(p,\frac{p-1}{k})$. In the special case when $k=2$, the generalized Paley graph is the Paley graph denoted by $\Paley(p)$.

\begin{theo}\label{main1}
A graph $\Gamma$ is a Deza circulant graph of WL-rank~$4$ if and only if it is isomorphic to one of the following graphs:
\\
\\
$(1)$ $K_4\times K_m$, where $m>1$ is odd;
\\
\\
$(2)$ $\overline{K_2\times K_m}$, where $m>1$ is odd;
\\
\\
$(3)$ $m\mathcal{C}_5$, where $m>1$;
\\
\\
$(4)$ $lK_{m,m}$, where $l,m>1$;
\\
\\
$(5)$ $K_l[mK_2]$, where $l,m>1$;
\\
\\
$(6)$ $\Paley(p)[K_2]$, where $p$ is a prime such that $p\equiv 1\mod~4$;
\\
\\
$(7)$ $\GPaley(p,\frac{p-1}{3})$, where $p$ is a prime of the form $p=t^2+3$  with an integer $t$;
\\
\\
$(8)$ $\overline{\GPaley(p,\frac{p-1}{3})}$, where $p$ is a prime of the form $p=t^2+12$  with an integer $t$.
\end{theo}
\vspace{\baselineskip}
Note that families of strictly Deza graphs from Theorem~\ref{main1} and their parameters were described in~\cite{GGSh}. However, the classification of Deza circulant graphs of WL-rank~$4$ was unknown. We find automorphism groups of graphs from Theorem~\ref{main1} and check whether these graphs are divisible design (see~\cite{HKM} for the definition). This information is presented in Appendix.

We study one more parameter of a graph that is closely related to the Weisfeiler-Leman algorithm. The \emph{WL-dimension} (the \emph{Weisfeiler-Leman dimension}) of a graph $\Gamma$ is defined to be the smallest positive integer~$m$ for which $\Gamma$ is identified by the $m$-dimensional Weisfeiler-Leman algorithm~~\cite[Definition~18.4.2]{Grohe}. The WL-dimension of $\Gamma$ is denoted by $\dimwl(\Gamma)$. The \emph{WL-dimension} of a class of graphs $\mathcal{K}$ is defined to be the number 
$$\dimwl(\mathcal{K})=\max \limits_{\Gamma\in \mathcal{K}} \dimwl(\Gamma).$$
Interest in the WL-dimension in recent years caused, in particular, by the fact that if $\dimwl(\mathcal{K})\leq m$ then the graph isomorphism problem restricted to $\mathcal{K}$ can be solved in polynomial time by the $m$-dimensional Weisfeiler-Leman algorithm. On the other hand, there exist infinitely many graphs with arbitrarily large WL-dimension~\cite{CFI}. For more information on WL-dimension of graphs, we refer the readers to~\cite{Grohe,GN}.

The isomorphism problem for circulant graphs was solved independently in~\cite{EP2,M}. However, the question on the WL-dimension of an arbitrary circulant graph remains open. We estimate the WL-dimension of graphs from Theorem~\ref{main1}.

\begin{theo}\label{main2}
The WL-dimension of the class of circulant Deza graphs of WL-rank~$4$ is equal to~$3$. 
\end{theo}

The WL-rank of a Deza circulant graph can be arbitrarily large. For example, the WL-rank of a cycle graph with $n$ vertices is equal to~$[\frac{n}{2}]+1$. Observe that a cycle graph is a non-strictly Deza graph. It seems that the situation may be different in case of strictly Deza circulant graphs. The analysis of the results of the computer calculations~\cite{GGSh,GSh} implies that every strictly Deza circulant graph with at most~$95$ vertices occurs in Theorem~\ref{main1} or belongs to one of the several families. We establish that the WL-rank of graphs from these families is equal to~$5$ or $6$ (Section~$9$). We also prove that the WL-dimension of the above graphs is at most~$3$.  So we obtain the following statement.

\begin{theo}\label{main3}
If $\Gamma$ is a strictly Deza circulant graph with at most~$95$ vertices then $\rkwl(\Gamma)\leq 6$ and $\dimwl(\Gamma)\leq 3$.
\end{theo}

The authors do not know strictly Deza circulant graphs which do not belong to one of the families from Theorem~\ref{main1} or Section~$9$. The WL-rank of each of the above graphs is at most~$6$. It would be interesting to investigate how large the WL-rank of a strictly Deza circulant graph can be. It would be also nice to classify all Deza circulant graphs whose WL-rank is equal to~$5$ or~$6$. Another natural question is how large the WL-dimension of a Deza circulant graph can be. Is it true that the WL-dimension of an arbitrary Deza circulant graph is at most~$3$?

We finish the introduction with the brief outline of the paper. Sections~$2$ and~$3$ define WL-rank and WL-dimension of graphs respectively. It is natural and convenient to study WL-rank of Cayley graphs using the theory of \emph{$S$-rings} (\emph{Schur rings}) that simplifies the theory of coherent configurations in this special case. Sections~$4$ and~$5$ are concerned with $S$-rings and Cayley graphs respectively. Section~$6$ is devoted to $S$-rings over cyclic groups. Section~$7$ contains number-theoretic results concerned with \emph{cyclotomic numbers} which appear as structure constants of cyclotomic $S$-rings. In Section~$8$ we prove Theorems~\ref{main1} and~\ref{main2}. In Section~$9$ we estimate the WL-rank and the WL-dimension of some strictly Deza circulant graphs and prove Theorem~\ref{main3}. Parameters, automorphism groups, and some properties of graphs from Theorem~\ref{main1} and Section~$9$  are collected in two tables in Appendix.

The authors are grateful to D. Churikov, Prof. I. Ponomarenko, and Prof. A. Vasil'ev for the fruitful discussions on the subject matters and the valuable comments which help us to improve the text significantly. The authors would like to thank Prof. V. Kabanov, Prof. E.~Konstantinova, and Dr. L.~Shalaginov who attracted the author's attention to the topic. We are also thankful to D.~Panasenko and all participants of Project No.~13 of The First Workshop at the Mathematical Center in Akademgorodok for the help with computer calculations.

\section{WL-rank}

Let $V$ be a finite set, $|V|=n$, and $\mathcal{R}$ a partition of $V\times V$. If $W\subseteq V$ then the diagonal of $W\times W$ is denoted by $1_W$. The set of all unions of elements from $\mathcal{R}$ is denoted by $\mathcal{R}^{\cup}$. If $\alpha\in V$ and $r\in \mathcal{R}$ then put $\alpha r=\{\beta\in V:~(\alpha,\beta)\in  r\}$. The pair $\mathcal{X}=(V,\mathcal{R})$ is called a \emph{coherent configuration} on $V$ if the following conditions are satisfied:

$(1)$ $1_{V}\in \mathcal{R}^{\cup}$;

$(2)$ if $r\in \mathcal{R}$ then $r^{*}=\{(\beta,\alpha):~(\alpha,\beta)\in r\} \in \mathcal{R}$;

$(3)$ given $r,s,t\in \mathcal{R}$, the number $c_{rs}^t=|\alpha r\cap \beta s^{*}|$ does not depend on the choice of $(\alpha,\beta)\in t$. 

\noindent The elements of $V$ and $\mathcal{R}$ are called the \emph{points} and \emph{basis relations} of $\mathcal{X}$ respectively. The number $\rk(\mathcal{X})=|\mathcal{R}|$ is called the \emph{rank} of $\mathcal{X}$. The numbers $c_{rs}^t$ are called the \emph{intersection numbers} of $\mathcal{X}$. If $1_{V}\in \mathcal{R}$ then $\mathcal{X}$ is called an (\emph{association}) \emph{scheme}.

One can define a partial order on the set of all coherent configurations on the same set~$V$. Given coherent configurations $\mathcal{X}$ and $\mathcal{X}^{\prime}$ on $V$, we set $\mathcal{X}\leq \mathcal{X}^{\prime}$ if and only if every basis relation of $\mathcal{X}$ is a union of some basis relations of $\mathcal{X}^{\prime}$.

Let $\Gamma=(V,E)$ be a digraph. The \emph{WL-closure} of $\Gamma$ is defined to be the smallest coherent configuration $\WL(\Gamma)=(V,\mathcal{R})$ on $V$ for which $E\in \mathcal{R}^{\cup}$. Note that for a given graph, the WL-closure can be found efficiently by using the Weisfeiler-Leman algorithm~\cite{WeisL}.

The \emph{WL-rank} $\rkwl(\Gamma)$ of $\Gamma$ is defined to be the rank of $\WL(\Gamma)$. Some basic properties of WL-rank are collected in the following lemma. 

\begin{lemm}\label{wlprop}
Let $\Gamma$ be a graph with $n$ vertices. The following statements hold:
\\
\\
$(1)$ $2\leq \rkwl(\Gamma)\leq n^2$;
\\
\\
$(2)$ if $\Gamma$ is vertex-transitive, in particular, if $\Gamma$ is a Cayley graph, then $2\leq \rkwl(\Gamma)\leq n$;
\\
\\
$(3)$ $\rkwl(\Gamma)=2$ if and only if $\Gamma$ is complete or empty;
\\
\\
$(4)$ if $\Gamma$ is a connected distance-regular graph of diameter~$d$ then $\rkwl(\Gamma)=d+1$. In particular, if $\Gamma$ is strongly regular then $\rkwl(\Gamma)\leq 3$.
\end{lemm} 

\begin{proof}
Statement~1 is obvious. If $\Gamma$ is vertex-transitive then the coherent configuration $\WL(\Gamma)$ is a scheme by~\cite[Corollary~2.2.6 (1)]{CP} and the rank of every scheme on $n$ points is at most~$n$. Therefore Statement~2 holds. Statement~3 follows from the definitions of WL-closure and WL-rank. Statement~4 follows from~\cite[Theorem~2.6.11(1)]{CP}.
\end{proof}

\section{WL-dimension}

The \emph{WL-dimension} of $\Gamma$ is defined to be the smallest positive integer~$m$ for which $\Gamma$ is identified by the $m$-dimensional Weisfeiler-Leman algorithm~~\cite[Definition~18.4.2]{Grohe}. The main goal of this section is to provide an approach to the estimation of WL-dimension of graphs. The key notion of this approach is the separability of coherent configurations.

Let $\mathcal{X}=(V,\mathcal{R})$ and $\mathcal{X}^{\prime}=(V^{\prime},\mathcal{R}^{\prime})$ be coherent configurations. An \emph{algebraic isomorphism} from $\mathcal{X}$ to $\mathcal{X}^{\prime}$ is defined to be a bijection $\varphi:\mathcal{R}\rightarrow \mathcal{R}^{\prime}$ such that $c_{rs}^t=c_{r^{\varphi}s^{\varphi}}^{t^{\varphi}}$ for every $r,s,t\in \mathcal{R}$.

An \emph{isomorphism} from $\mathcal{X}$ to $\mathcal{X}^{\prime}$ is defined to be a bijection $f:V\rightarrow V^{\prime}$ such that $\mathcal{R}^{\prime}=\mathcal{R}^f$, where $\mathcal{R}^f=\{r^f:~r\in \mathcal{R}\}$ and $r^f=\{(\alpha^f,\beta^f):~(\alpha,\beta)\in r\}$. The group of all isomorphisms from $\mathcal{X}$ onto itself has a normal subgroup
$$\aut(\mathcal{X})=\{f\in \sym(V): r^f=r~\text{for every}~r\in \mathcal{R}\}.$$
called the \emph{automorphism group} of $\mathcal{X}$. If $\aut(\mathcal{X})$ is transitive then $\mathcal{X}$ is a scheme. If $\Gamma$ is a digraph then 
$$\aut(\Gamma)=\aut(\WL(\Gamma)).~\eqno(1)$$

Every isomorphism of coherent configurations induces in a natural way the algebraic isomorphism of them. However, not every algebraic isomorphism is induced by a combinatorial one (see~\cite[Section~2.3.4]{CP}). A coherent configuration is called \emph{separable} if every algebraic isomorphism from it to another coherent configuration is induced by an isomorphism. 

Let $\alpha\in V$. The \emph{one-point extension} or \emph{$\alpha$-extension} $\mathcal{X}_{\alpha}=(V,\mathcal{R}_{\alpha})$ of $\mathcal{X}$ is defined to be the smallest coherent configuration on $V$ such that $\mathcal{R}\subseteq \mathcal{R}_{\alpha}^{\cup}$ and $\{(\alpha,\alpha)\}\in \mathcal{R}_{\alpha}$.

\begin{lemm}\label{wldim}
Let $\Gamma$ be a digraph and $\mathcal{X}=\WL(\Gamma)$. The following statements hold:

$(1)$ $\dimwl(\Gamma)\leq 2$ if and only if $\mathcal{X}$ is separable;

$(2)$ if $\mathcal{X}_{\alpha}$ is separable for some $\alpha\in V$ then $\dimwl(\Gamma)\leq 3$.
\end{lemm}

\begin{proof}
The first statement of the lemma follows from~\cite[Theorem~2.5]{FKV}. Let us prove the second one. Note that $\mathcal{X}_{\alpha}=\WL(\Gamma_{\alpha})$, where $\Gamma_{\alpha}$ is the digraph obtained from $\Gamma$ by the individualization of the vertex~$\alpha$. Since $\mathcal{X}_{\alpha}$ is separable, \cite[Theorem~2.5]{FKV} implies that $\dimwl(\Gamma_{\alpha})\leq 2$. Therefore, $\dimwl(\Gamma)\leq \dimwl(\Gamma_{\alpha})+1\leq 3$ by~\cite[Eq.~(41)]{P}.
\end{proof}

\begin{lemm}\label{dim1}
If $\Gamma$ is a regular graph such that $\dimwl(\Gamma)=1$ then $\Gamma$ is strongly regular, i.e. $\rkwl(\Gamma)\leq 3$.
\end{lemm}

\begin{proof}
The statement of the lemma follows from the description of all regular graphs of WL-dimension~$1$~\cite[Lemma~3.1 (a)]{AKRV}.
\end{proof}

\section{$S$-rings}

If $\Gamma$ is a Cayley graph over a group $G$ then the WL-closure of $\Gamma$ can be thought as a special subring of the group ring $\mathbb{Z}G$, so-called $S$-ring. A using of the $S$-ring theory  simplifies substantially the studying of WL-closures of Cayley graphs. In this section we provide a background of $S$-rings. We use the notations and terminology from~\cite{Ry}.

\subsection{Definitions}
Let $G$ be a finite group and $\mathbb{Z}G$  the integer group ring. The identity element of $G$ is denoted by~$e$. The set of all of non-identity elements of $G$ is denoted by $G^\#$. If $X\subseteq G$ then the element $\sum_{x\in X} {x}$ of the group ring $\mathbb{Z}G$ is denoted by $\underline{X}$. The set $\{x^{-1}:x\in X\}$ is denoted by $X^{-1}$.

A subring  $\mathcal{A}\subseteq \mathbb{Z} G$ is called an \emph{$S$-ring} (a \emph{Schur} ring) over $G$ if there exists a partition $\mathcal{S}=\mathcal{S}(\mathcal{A})$ of~$G$ such that:

$(1)$ $\{e\}\in\mathcal{S}$;

$(2)$  if $X\in\mathcal{S}$ then $X^{-1}\in\mathcal{S}$;

$(3)$ $\mathcal{A}=\Span_{\mathbb{Z}}\{\underline{X}:\ X\in\mathcal{S}\}$.

\noindent The notion of an $S$-ring goes back to I.~Schur~\cite{Schur} and H.~Wielandt~\cite{Wi}. The elements of $\mathcal{S}$ are called the \emph{basic sets} of  $\mathcal{A}$ and the number $\rk(\mathcal{A})=|\mathcal{S}|$ is called the \emph{rank} of~$\mathcal{A}$. The $S$-ring of rank~$2$ over $G$ is denoted by $\mathcal{T}_G$. Clearly, $\mathcal{S}(\mathcal{T}_G)=\{\{e\},G^\#\}$. If $X,Y\in \mathcal{S}(\mathcal{A})$ then $XY\in \mathcal{S}(\mathcal{A})$ whenever $|X|=1$ or $|Y|=1$.

Let $X,Y\in\mathcal{S}$. If $Z\in \mathcal{S}$ then the number of distinct representations of $z\in Z$ in the form $z=xy$ with $x\in X$ and $y\in Y$ does not depend on the choice of $z\in Z$. Denote this number by $c^Z_{XY}$. One can see that $\underline{X}~\underline{Y}=\sum_{Z\in \mathcal{S}(\mathcal{A})}c^Z_{XY}\underline{Z}$. Therefore the numbers  $c^Z_{XY}$ are the structure constants of $\mathcal{A}$ with respect to the basis $\{\underline{X}:\ X\in\mathcal{S}\}$. From~\cite[Eq.~(2.1.14)]{CP} it follows that 
$$|Z|c^{Z^{-1}}_{XY}=|X|c^{X^{-1}}_{YZ}=|Y|c^{Y^{-1}}_{ZX}~\eqno(2)$$
for all $X,Y,Z\in \mathcal{S}(\mathcal{A})$. 

A set $X \subseteq G$ is called an \emph{$\mathcal{A}$-set} if $\underline{X}\in \mathcal{A}$. A subgroup $H \leq G$ is called an \emph{$\mathcal{A}$-subgroup} if $H$ is an $\mathcal{A}$-set. An $\mathcal{A}$-subgroup $H$ is called \emph{nontrivial} if $\{e\}<H<G$. For a set $X\subseteq G$, put $\rad(X)=\{g\in G:\ gX=Xg=X\}$. One can check that for every $\mathcal{A}$-set $X$, the groups $\langle X \rangle$ and $\rad(X)$ are $\mathcal{A}$-subgroups. The $S$-ring $\mathcal{A}$ is called \emph{primitive} if there are no nontrivial proper $\mathcal{A}$-subgroups of $G$.

Let $L\leq U$ be $\mathcal{A}$-subgroups. The straightforward computation shows that
$$(\underline{U}-\underline{L})^2=(|U|-2|L|)\underline{U\setminus L}+(|U|-|L|)\underline{L}.~\eqno(3)$$

Let $L \unlhd U\leq G$. A section $U/L$ is called an \emph{$\mathcal{A}$-section} if $U$ and $L$ are $\mathcal{A}$-subgroups. If $S=U/L$ is an $\mathcal{A}$-section then the module
$$\mathcal{A}_S=Span_{\mathbb{Z}}\left\{\underline{X}^{\pi}:~X\in\mathcal{S}(\mathcal{A}),~X\subseteq U\right\},$$
where $\pi:U\rightarrow U/L$ is the canonical epimorphism, is an $S$-ring over $S$.

Let $K \leq \aut(G)$. The partition of $G$ into the orbits of $K$ defines an  $S$-ring $\mathcal{A}$ over~$G$.  In this case  $\mathcal{A}$ is called \emph{cyclotomic} and denoted by $\cyc(K,G)$. Let $G\cong C_p$ for some prime~$p$. The group $\aut(G)$ is a cyclic group of order~$p-1$. So for every divisor $k$ of $p-1$ there exists a unique subgroup of $\aut(G)$ of order~$k$. For a given divisor $k$ of $p-1$, put $\cyc(k,G)=\cyc(K,G)$, where $K\leq \aut(G)$ such that $|K|=k$.

\subsection{$S$-rings and Cayley schemes}
 A scheme $\mathcal{X}=(G,\mathcal{R})$ is called a \emph{Cayley scheme} over $G$ if $\aut(\mathcal{X})\geq G_{right}$, where $G_{right}$ is the subgroup of $\sym(G)$ induced by the right multiplications of $G$. The Cayley scheme $\mathcal{X}$ is called \emph{normal} if $G_{right}$ is normal in $\aut(\mathcal{X})$.

There is a one-to-one correspondence between  $S$-rings and Cayley schemes over $G$. Namely, if $\mathcal{A}$ is an $S$-ring over $G$ then the  pair 
$$\mathcal{X}(\mathcal{A})=(G,\mathcal{R}(\mathcal{A})),$$ 
where $\mathcal{R}(\mathcal{A})=\{r(X):X\in \mathcal{S}(\mathcal{A})\}$ and $r(X)=\{(g,xg): g\in G, x\in X\}$, is a Cayley scheme over~$G$. Conversely, if $\mathcal{X}=(G,\mathcal{R})$ is a Cayley scheme over $G$ then 
$$\mathcal{S}(\mathcal{X})=\{X(r):r\in \mathcal{R}\},$$ 
where $X(r)=\{x\in G: (e,x) \in r\}\subseteq G$, is a partition of $G$ that defines the $S$-ring $\mathcal{A}(\mathcal{X})$ over $G$. Clearly, 
$$\rk(\mathcal{X})=\rk(\mathcal{A}(\mathcal{X})).$$ 
For all $X,Y,Z\in \mathcal{S}(\mathcal{A})$, we have $c_{XY}^Z=c_{r(X)r(Y)}^{r(Z)}$. 

The \emph{automorphism group} $\aut(\mathcal{A})$ of the $S$-ring $\mathcal{A}$ is defined to be the group $\aut(\mathcal{X}(\mathcal{A}))$. It is easy to see that $\aut(\mathcal{T}_G)=\sym(G)$. We say that $\mathcal{A}$ is \emph{separable} (respectively, \emph{normal}) if $\mathcal{X}(\mathcal{A})$ is separable (respectively, normal). The $S$-rings $\mathcal{T}_G$ and $\mathbb{Z}G$ are separable.

\subsection{Generalized wreath and tensor products}

Let $S=U/L$ be an $\mathcal{A}$-section of $G$. The $S$-ring~$\mathcal{A}$ is called the \emph{$S$-wreath product} or \emph{generalized wreath product} of $\mathcal{A}_U$ and $\mathcal{A}_{G/L}$ if $L\trianglelefteq G$ and $L\leq\rad(X)$ for each basic set $X$ outside~$U$. In this case we write $\mathcal{A}=\mathcal{A}_U\wr_{S}\mathcal{A}_{G/L}$. If $L>\{e\}$ and $U<G$ then the $S$-wreath product is called \emph{nontrivial}. If $U=L$ then  $\mathcal{A}$ coincides with the \emph{wreath product} of $\mathcal{A}_L$ and $\mathcal{A}_{G/L}$ denoted by $\mathcal{A}_L\wr\mathcal{A}_{G/L}$. From~\cite[Eq.~3.4.11]{CP} it follows that 
$$\aut(\mathcal{A}_L\wr \mathcal{A}_{G/L})=\aut(\mathcal{A}_L)\wr \aut(\mathcal{A}_{G/L}).~\eqno(4)$$

Let $H$ and $V$ be $\mathcal{A}$-subgroups such that $G=H\times V$. The $S$-ring~$\mathcal{A}$ is called the \emph{tensor product} of $\mathcal{A}_H$ and $\mathcal{A}_V$ if 
$$\mathcal{S}(\mathcal{A})=\{X_1\times X_2:~X_1\in\mathcal{S}(\mathcal{A}_H),~X_2\in \mathcal{S}(\mathcal{A}_V)\}.$$
In this case we write $\mathcal{A}=\mathcal{A}_H\otimes \mathcal{A}_V$. The tensor product is called \emph{nontrivial} if $\{e\}<H<G$ and $\{e\}<V<G$. It is easy to see that $\rk(\mathcal{A}_H\otimes \mathcal{A}_V)=\rk(\mathcal{A}_H)\rk(\mathcal{A}_V)$. From~\cite[Eq.~3.2.18]{CP} it follows that  
$$\aut(\mathcal{A}_H\otimes \mathcal{A}_V)=\aut(\mathcal{A}_H)\times \aut(\mathcal{A}_V).~\eqno(5)$$

\begin{lemm}\label{separ}
Let $\mathcal{A}_1$ and $\mathcal{A}_2$ be $S$-rings. The $S$-ring $\mathcal{A}_1*\mathcal{A}_2$, where $*\in\{\wr,\otimes\}$, is separable if and only if so are $\mathcal{A}_1$ and $\mathcal{A}_2$.
\end{lemm}
\begin{proof}
The statement of the lemma follows from~\cite[Corollary~3.2.24, Theorem~3.4.9]{CP}.
\end{proof}

\section{Cayley graphs}

Let $S\subseteq G$ with $e\notin S$, $\Gamma=\cay(G,S)$ a Cayley digraph over $G$, and $\mathcal{X}=\WL(\Gamma)$. Since $\aut(\Gamma)\geq G_{right}$ and, due to Eq.~(1), $\aut(\mathcal{X})=\aut(\Gamma)$, we conclude that $\mathcal{X}$ is a Cayley scheme. Let $\mathcal{A}=\mathcal{A}(\mathcal{X})$. One can check that $\mathcal{A}$ is the smallest $S$-ring over $G$ for which $S$ is an $\mathcal{A}$-set. So we obtain the following statement.

\begin{lemm}\label{rankring}
Let $G$ be a finite group, $S\subseteq G$ with $e\notin S$, and $\Gamma=\cay(G,S)$. Suppose that $\mathcal{A}$ is the smallest $S$-ring over $G$ for which $S$ is an $\mathcal{A}$-set. Then $\rkwl(\Gamma)=\rk(\mathcal{A})$.
\end{lemm} 

Further throughout the paper for a given Cayley digraph $\Gamma=\cay(G,S)$, we will denote by $\WL(\Gamma)$ the smallest $S$-ring $\mathcal{A}$ over $G$ for which $S$ is an $\mathcal{A}$-set. If $T\subseteq G$ then the smallest $S$-ring $\mathcal{A}$ such that $T$ is an $\mathcal{A}$-set is denoted by $\WL(T)$.

The next statement provides a criterion for a Cayley graph to be a Deza graph. 

\begin{lemm}\label{deza}
Let $G$ be a group of order~$n$, $S\subseteq G$ such that $e\notin S$, $S=S^{-1}$, and $|S|=k$, and $\Gamma=\cay(G,S)$. The graph $\Gamma$ is a Deza graph with parameters $(n,k,b,a)$ if and only if $\underline{S}^2=a\underline{A}+b\underline{B}+ke$, where $A\cup B\cup \{e\}=G$ and $A\cap B=\varnothing$. Moreover, $\Gamma$ is strongly regular if and only if $A=S$ or $B=S$.
\end{lemm}

\begin{proof}
Note that $\underline{S}^2=a\underline{A}+b\underline{B}+ke$ if and only if the multiset  $SS=\{s_1s_2:~s_1,s_2\in S\}$ contains $a$ copies of each element of $A$, $b$ copies of each element of $B$, and $k$ copies of $e$. So the statement of the lemma follows from~\cite[Proposition~2.1]{EFHHH}.
\end{proof}

Let $p$ be a prime and $k$ a divisor of $p-1$ such that $\frac{p-1}{k}$ is even. The group $\aut(C_p)$ contains a unique subgroup $K$ of order $\frac{p-1}{k}$. Let $S$ be an orbit of $K$ and $\Gamma=\cay(G,S)$. One can check that $\Gamma$ is isomorphic to the \emph{generalized Paley graph} $\GPaley(p,\frac{p-1}{k})$ (see~\cite[Definition~1.1]{LP}). If $k=2$ then $\Gamma$ is isomorphic to the \emph{Paley graph} $\Paley(p)$.

Let $\Gamma_1=(V_1,E_1)$ and $\Gamma_2=(V_2,E_2)$ be graphs. The \emph{tensor product} $\Gamma_1\times \Gamma_2$ of $\Gamma_1$ and $\Gamma_2$ is defined to be the graph with vertex set $V=V_1\times V_2$ and edge set $E$ defined as follows:
$$((v_1,v_2),(u_1,u_2))\in E~\text{if and only if}~(v_1,u_1)\in E_1,~(v_2,u_2)\in E_2.$$
The next lemma follows from the definitions of the Cayley graph and the tensor product of graphs.

\begin{lemm}\label{tensproduct}
Let $\Gamma_1=\cay(G_1,S_1)$ and $\Gamma_2=\cay(G_2,S_2)$ be Cayley graphs over groups $G_1$ and $G_2$ respectively. Then 
$$\Gamma_1\times \Gamma_2=\cay(G_1\times G_2,S_1\times S_2).$$
\end{lemm}

Again, let $\Gamma_1=(V_1,E_1)$ and $\Gamma_2=(V_2,E_2)$ be graphs. The \emph{lexicographic product} $\Gamma_1[\Gamma_2]$ of $\Gamma_1$ and $\Gamma_2$ is defined to be the graph with vertex set $V=V_1\times V_2$ and edge set $E$ defined as follows:
$$((v_1,v_2),(u_1,u_2))\in E~\text{if and only if}~(v_1,u_1)\in E_1~\text{or}~v_1=u_1~\text{and}~(v_2,u_2)\in E_2.$$

\begin{lemm}\label{lexproduct}
Let $G$ be a group, $H\trianglelefteq G$, $\pi:G\rightarrow G/H$ the canonical epimorphism, and $\overline{G}=G^{\pi}$. Suppose that $\Gamma_1=\cay(\overline{G},\overline{T})$ and $\Gamma_2=\cay(H,S)$ are Cayley graphs over $\overline{G}$ and $H$ respectively. Then 
$$\Gamma_1[\Gamma_2]\cong \cay(G,(\overline{T})^{\pi^{-1}}\cup S).$$
\end{lemm}

\begin{proof}
Let $\Gamma=\cay(G,(\overline{T})^{\pi^{-1}}\cup S)$, $|G:H|=m$, and $Hg_0=H,Hg_1,\ldots,Hg_{m-1}$ the pairwise distinct right cosets of $G$ by $H$. One can see that two vertices $x,y$ of $\Gamma$ are adjacent if and only if $Hx$ and $Hy$ are adjacent in $\Gamma_1$ or $Hx=Hy$ and $x$ and $y$ are adjacent in the induced subgraph $\Gamma_{Hx}$ which is isomorphic to $\Gamma_2$. Thus, the bijection $\varphi:G\rightarrow (G/H) \times H$ such that $g^{\varphi}=(Hg_i,gg_i^{-1})$, where $Hg_i$ is the right $H$-coset containing $g$, is an isomorphism from $\Gamma$ to $\Gamma_1[\Gamma_2]$.
\end{proof}

\begin{lemm}\label{disjointunion}
Let $G$ be a group and $H\leq G$. Suppose that $\Gamma=\cay(H,S)$ is a Cayley graph over $H$. Then $\Gamma^{\prime}=\cay(G,S)$ is isomorphic to a disjoint union of $m$ copies of $\Gamma$, where $m=|G:H|$.
\end{lemm}
\begin{proof}
Let $Hg_0=H,Hg_1,\ldots,Hg_{m-1}$ be pairwise distinct right cosets of $G$ by $H$. It is easy to see that there are no edges between $Hg_i$ and $Hg_j$ for all $i,j$ with $i\neq j$ and the induced subgraph $\Gamma^{\prime}_{Hg_i}$ is isomorphic to $\Gamma$ for every $i$.
\end{proof}

\section{Circulant $S$-rings of rank~$4$}
In this section we classify $S$-rings of rank~$4$ over cyclic groups. The main result of the section can be formulated as follows. 

\begin{theo}\label{rank4}
Let $\mathcal{A}$ be an $S$-ring of rank~$4$ over a cyclic group $G$. Then one of the following statements holds:
\\
\\
$(1)$ $\mathcal{A}=\mathcal{T}_L \otimes \mathcal{T}_U$, where $L$ and $U$ are nontrivial $\mathcal{A}$-subgroups such that $G=L\times U$;
\\
\\
$(2)$ $\mathcal{A}=(\mathcal{T}_L\wr \mathcal{T}_{U/L})\wr \mathcal{T}_{G/U}$ for some nontrivial $\mathcal{A}$-subgroups $L$ and $U$ such that $L<U$;
\\
\\
$(3)$ $\mathcal{A}=\cyc(\frac{p-1}{2},L)\wr \mathcal{T}_{G/L}$ for some nontrivial $\mathcal{A}$-subgroup $L$ of prime order $p\geq 3$;
\\
\\
$(4)$ $\mathcal{A}=\mathcal{T}_L\wr \cyc(\frac{p-1}{2},G/L)$ for some nontrivial $\mathcal{A}$-subgroup $L$ of prime index $p\geq 3$;
\\
\\
$(5)$ $|G|=p$ for some prime $p$ such that $p\equiv 1\mod~3$ and $\mathcal{A}=\cyc(\frac{p-1}{3},G)$;
\\
\\
$(6)$ $|G|=4$ and $\mathcal{A}=\mathbb{Z}G$.
\end{theo}

Note that families of $S$-rings from Statements~1-6 of Theorem~\ref{rank4} are disjoint. The proof of Theorem~\ref{rank4} will be given later. We start with well-known Wielandt's theorem on primitive $S$-rings over cyclic groups.

\begin{lemm}\cite[Theorem~25.3]{Wi}\label{primitive}
If $\mathcal{A}$ is a primitive $S$-ring over a cyclic group of composite order then $\rk(\mathcal{A})=2$.
\end{lemm}

If $\mathcal{A}$ is an $S$-ring over a cyclic group $G$ then put $\rad(\mathcal{A})=\rad(X)$, where $X$ is a basic set of $\mathcal{A}$ containing a generator of $G$. One can check that $\rad(\mathcal{A})$ does not depend on the choice of $X$. $S$-rings over cyclic groups were described for the first time in~\cite{LeungMan1,LeungMan2}. We provide a description of $S$-rings over cyclic groups in a convenient for us form. 

\begin{lemm}\label{circ}
Let $\mathcal{A}$ be an $S$-ring over a cyclic group $G$. Then one of the following statements holds:
\\
\\
$(1)$ $\mathcal{A}=\mathcal{T}(G)$;
\\
\\
$(2)$ $\mathcal{A}$ is the nontrivial tensor product of two $S$-rings over proper subgroups of $G$;
\\
\\
$(3)$ $\mathcal{A}$ is the nontrivial $S$-wreath product for some $\mathcal{A}$-section $S$;
\\
\\
$(4)$ $\mathcal{A}$ is a normal cyclotomic $S$-ring with $|\rad(\mathcal{A})|=1$.
\end{lemm}

\begin{proof}
The statement of the lemma follows from~\cite[Theorem~4.1, Theorem~4.2]{EP3}.
\end{proof}

\begin{lemm}\label{rank3}
Let $\mathcal{A}$ be an $S$-ring of rank~$3$ over a cyclic group $G$. Then one of the following statements holds:
\\
\\
$(1)$ $\mathcal{A}=\mathcal{T}_L\wr \mathcal{T}_{G/L}$ for some nontrivial $\mathcal{A}$-subgroup $L$;
\\
\\
$(2)$ $|G|=p$ for some prime $p\geq 3$ and $\mathcal{A}=\cyc(\frac{p-1}{2},G)$.
\end{lemm}

\begin{proof}
Let $\mathcal{S}(\mathcal{A})=\{X_0,X_1,X_2\}$, where $X_0=\{e\}$. Suppose that $\mathcal{A}$ is imprimitive. Then without loss of generality we may assume that $L=X_0\cup X_1$ is a subgroup of $G$. This implies that $X_2=G\setminus L$ and hence $L\leq \rad(X_2)$. Thus, Statement~1 of Lemma~\ref{rank3} holds for $\mathcal{A}$. 

Now suppose that $\mathcal{A}$ is primitive. Then $G\cong C_p$ for some prime $p$ by Lemma~\ref{primitive}. So $\mathcal{A}$ is not a tensor or generalized wreath product of two $S$-rings. Therefore $\mathcal{A}=\cyc(K,G)$ for some $K\leq \aut(G)$ by Lemma~\ref{circ}. Since $\rk(\mathcal{A})=3$, we conclude that Statement~2 of Lemma~\ref{rank3} holds for $\mathcal{A}$.
\end{proof}

\begin{proof}[Proof of Theorem~\ref{rank4}]
One of the Statements~2-4 of Lemma~\ref{circ} holds for $\mathcal{A}$. Suppose that $\mathcal{A}=\mathcal{A}_L \otimes \mathcal{A}_U$, where $L$ and $U$ are nontrivial $\mathcal{A}$-subgroups such that $G=L\times U$. Since $\rk(\mathcal{A}_L)\rk(\mathcal{A}_U)=\rk(\mathcal{A})=4$, we obtain $\rk(\mathcal{A}_L)=\rk(\mathcal{A}_L)=2$ and hence $\mathcal{A}_L=\mathcal{T}_L$ and $\mathcal{A}_U=\mathcal{T}_U$. So Statement~1 of Theorem~\ref{rank4} holds for $\mathcal{A}$.

Suppose that $\mathcal{A}$ is the nontrivial $S$-wreath product for some $\mathcal{A}$-section $S=U/L$. Clearly, $\rk(\mathcal{A}_U)\in\{2,3\}$. If $\rk(\mathcal{A}_U)=2$ then $U=L$ and hence $\mathcal{A}=\mathcal{T}_L\wr \mathcal{A}_{G/L}$. The $S$-ring $\mathcal{A}_{G/L}$ has rank~$3$. So Statement~1 or Statement~2 of Lemma~\ref{rank3} holds for $\mathcal{A}_{G/L}$. In the former case Statement~2 of Theorem~\ref{rank4} holds for $\mathcal{A}$ ; in the latter case Statement~4 of Theorem~\ref{rank4} holds for $\mathcal{A}$. 

If $\rk(\mathcal{A}_U)=3$ then $G\setminus U\in \mathcal{S}(\mathcal{A})$ and hence $\mathcal{A}=\mathcal{A}_U\wr \mathcal{T}_{G/U}$. Statement~1 or Statement~2 of Lemma~\ref{rank3} holds for $\mathcal{A}_U$. In the former case Statement~2 of Theorem~\ref{rank4} holds for $\mathcal{A}$; in the latter case Statement~3 of Theorem~\ref{rank4} holds for $\mathcal{A}$.

Now suppose that $\mathcal{A}=\cyc(K,G)$ for some $K\leq \aut(G)$ and $|\rad(\mathcal{A})|=1$. Each basic set consists of elements of the same order because $\mathcal{A}$ is cyclotomic. Since $\rk(\mathcal{A})=4$, the total number of divisors of $|G|$ (including $1$ and $|G|$) is at most~$4$. Therefore
$$|G|=pq~\text{or}~|G|=p^k,$$
where $p$ and $q$ are distinct primes and $k\in\{1,2,3\}$. In the former case $\mathcal{A}=\mathcal{T}_P\otimes \mathcal{T}_Q$, where $P$ and $Q$ are subgroups of $G$ of orders $p$ and $q$ respectively, because $\rk(\mathcal{A})=4$. So Statement~1 of Theorem~\ref{rank4} holds for $\mathcal{A}$. In the latter case from~\cite[Lemma~5.1]{EP1} it follows that $|K|\leq p-1$ whenever $p$ is odd and $|K|\leq 2$ whenever $p=2$. This yields that 
$$4=\rk(\mathcal{A})\geq (p^k-1)/(p-1)+1~\eqno(6)$$
if $p$ is odd, and
$$4=\rk(\mathcal{A})\geq (2^k-1)/2+1~\eqno(7)$$
if $p=2$. If $p$ is odd then Eq.~(6) implies that $k=1$. In this case $|\aut(G):K|=3$ because $\rk(\mathcal{A})=4$ and Statement~5 of Theorem~\ref{rank4} holds for $\mathcal{A}$. If $p=2$ then due to Eq.~(7) and $\rk(\mathcal{A})=4$, we obtain $k=2$, i.e. $G\cong C_4$. Clearly, in this case $|K|=1$ and Statement~6 of Theorem~\ref{rank4} holds for $\mathcal{A}$.
\end{proof}

\begin{lemm}\label{autnorm}
Let $\mathcal{A}$ be a normal cyclotomic $S$-ring with trivial radical over a cyclic group $G$ and $K\leq \aut(G)$ such that $\mathcal{A}=\cyc(K,G)$. Then $\aut(\mathcal{A})=G_{right}\rtimes K$.
\end{lemm}

\begin{proof}
By the definition of the automorphism group, we have $\aut(\mathcal{A})\geq G_{right}\rtimes K$ and hence $\aut(\mathcal{A})_e\geq K$, where $\aut(\mathcal{A})_e$ is the stabilizer of $e$ in $\aut(\mathcal{A})$. From~\cite[Theorem~4.7]{EP1} it follows that $\aut(\mathcal{A})_e\leq \aut(G)$. Note that $|\aut(\mathcal{A})_e|=|X|=|K|$, where $X$ is a basic set of $\mathcal{A}$ containing a generator of $G$, because $G$ is cyclic, $K\leq \aut(G)$, and $\aut(\mathcal{A})_e\leq \aut(G)$. Thus, $\aut(\mathcal{A})_e=K$ and hence $\aut(\mathcal{A})=G_{right}\rtimes K$.
\end{proof}

In the proofs of the next three lemmas we use freely some notions concerned with coherent configurations. All of them can be found in~\cite{CP}.

\begin{lemm}\label{onepointnormal}
Let $\Gamma$ be a Cayley digraph over a cyclic group $G$ and $\mathcal{A}=\WL(\Gamma)$. Suppose that $\mathcal{A}$ is a normal cyclotomic $S$-ring with trivial radical. Then $\dimwl(\Gamma)\leq 3$.
\end{lemm}
\begin{proof}
Let $\mathcal{X}=\mathcal{X}(\mathcal{A})$ and $g\in G$. From~\cite[Theorem~4.4.7]{CP} it follows that the one-point-extension $\mathcal{X}_g$ is partly regular. So $\mathcal{X}_g$ is separable by~\cite[Theorem~3.3.19]{CP}. Therefore $\dimwl(\Gamma)\leq 3$ by Statement~2 of Lemma~\ref{wldim}.
\end{proof}

A \emph{parabolic} $E$ of a coherent configuration $\mathcal{X}=(V,\mathcal{R})$ is an equivalence relation on $V$ such that $E\in\mathcal{R}^{\cup}$. The quotient $\mathcal{X}$ modulo $E$ is denoted by $\mathcal{X}_{V/E}$ (see~\cite[Section~3.1.2]{CP} for the definition). 

\begin{lemm}\label{ineq}
If $\mathcal{X}$ is a coherent configuration on a set $V$, $\alpha\in V$, $E$ a parabolic of $\mathcal{X}$, and $\Delta$ the class of $E$ containing $\alpha$ then $(\mathcal{X}_{\alpha})_{V/E}\geq (\mathcal{X}_{V/E})_{\Delta}$. 
\end{lemm}

\begin{proof}
The coherent configuration $(\mathcal{X}_{\alpha})_{V/E}$ satisfies the following conditions: (1) $(\mathcal{X}_{\alpha})_{V/E}\geq \mathcal{X}_{V/E}$ because $\mathcal{X}_{\alpha}\geq \mathcal{X}$; (2) $\{(\Delta,\Delta)\}$ is a basis relation of $(\mathcal{X}_{\alpha})_{V/E}$ because $\{(\alpha,\alpha)\}$ is a basis relation of $\mathcal{X}_{\alpha}$. By the definition, the coherent configuration $(\mathcal{X}_{V/E})_{\Delta}$ is the smallest coherent configuration satisfying two above conditions and hence the required statement holds.
\end{proof}

\begin{lemm}\label{paley}
Let $p$ be an odd prime, $G\cong C_{2p}$, $\Gamma$ a Cayley digraph over $G$, and $\mathcal{A}=\WL(\Gamma)$. Suppose that $\mathcal{A}=\mathcal{T}_L\wr \cyc(\frac{p-1}{2},G/L)$, where $L$ is the subgroup of $G$ of order~$2$. Then $\dimwl(\Gamma)\leq 3$.
\end{lemm}

\begin{proof}
Let $\mathcal{X}=\mathcal{X}(\mathcal{A})$, $E=E(L)$ the parabolic of $\mathcal{X}$ corresponding to $L$, and $g\in G$. Note that $E$ is a parabolic of the one-point extension $\mathcal{X}_g$ because $\mathcal{X}_g\geq \mathcal{X}$. Lemma~\ref{ineq} implies that 
$$(\mathcal{X}_g)_{G/E}\geq (\mathcal{X}_{G/E})_{Lg}.~\eqno(8)$$

Since $|G/L|=p$ is prime and $\rk(\mathcal{A}_{G/L})=3$, Statement~$4$ of Lemma~\ref{circ} holds for $\mathcal{A}_{G/L}$, i.e. $\mathcal{A}_{G/L}$ is a normal cyclotomic $S$-ring with trivial radical. Let $K\leq \aut(G/L)$ such that $\mathcal{A}_{G/L}=\cyc(K,G/L)$. Clearly, $\mathcal{X}_{G/E}=\mathcal{X}(\mathcal{A}_{G/L})$. From~\cite[Theorem~4.4.9]{CP} applied to $G/L$ and $\mathcal{X}_{G/E}$ it follows that the basis relations of $(\mathcal{X}_{G/E})_{Lg}$ are orbits of $K$ acting on $G/L\times G/L$ componentwise. The group $K$ is semiregular because $K\leq \aut(G/L)$ and $|G/L|=p$ is prime. So $(\mathcal{X}_{G/E})_{Lg}$ is semiregular, i.e. all its basis relations have valency~$1$. Eq.~(8) yields that  $(\mathcal{X}_g)_{G/E}$ is also semiregular.

Now due to~\cite[Theorem~3.1.28~(1)]{CP}, to prove that $\mathcal{X}_g$ is separable, it is sufficient to prove that the extension $(\mathcal{X}_g)_E$ of $\mathcal{X}_g$ with respect to $E$ is separable. Each class of $E$ is an $L$-coset and hence it has size~$2$. This implies that each fiber of $(\mathcal{X}_g)_E$ has size at most~$2$. Therefore $(\mathcal{X}_g)_E$  is separable by~\cite[Exercise~3.7.20]{CP}. Thus, $\mathcal{X}_g$ is separable. In view of Statement~2 of Lemma~\ref{wldim}, we obtain $\dimwl(\Gamma)\leq 3$.
\end{proof}

\section{Cyclotomic numbers}
Let $m$ and $f$ be nonnegative integers, $p=mf+1$ an odd prime, and $l$ a fixed primitive root of $p$. The \emph{cyclotomic number} $(k,h)_m$ of order~$m$ is the number of solutions $x,y$ of the congruence
$$l^{mx+k}+1 \equiv l^{my+h} \mod p,$$
where the integers $x,y$ are chosen from $\{0,\ldots,f-1\}$~\cite{Di}.

The main goal of this section is to establish a relation between cyclotomic numbers and structure constants of a cyclotomic $S$-ring over a group of prime order. Throughout the section, we follow to~\cite{Di}, where the most part of the material is contained.

Let $G\cong C_p$ and $a$ a generator of $G$. The group $\aut(G)$ is a cyclic group of order $p-1$. The element $\sigma$ such that $a^{\sigma}=a^l$ is a generator of $\aut(G)$. Suppose that $p-1$ is divisible by $m$
and $\mathcal{A}=\cyc(\frac{p-1}{m},G)$. Then $\mathcal{S}(\mathcal{A})=\{X_0,\ldots, X_m\}$, where $X_0=\{e\}$ and
$$X_i=\{a^{l^{mj+(i-1)}}:j=0,\ldots,f-1\},~i\geq 1.$$
Put $c_{ij}^k=c_{X_iX_j}^{X_k}$. Due to the definitions of a cyclotomic number and $X_i$,  we obtain
$$c_{ij}^1=|aX_i^{-1}\cap X_j|=|\{a^s:~s\equiv l^{mx+(k-1)}+1 \equiv l^{my+(j-1)} \mod p\}|=(k-1,j-1)_m,~\eqno(9)$$
where $X_k=X_i^{-1}$, $i,j,k\geq 1$, and $x,y$ are chosen from $\{0,\ldots,f-1\}$.

\begin{lemm}\label{cyc1}
In the above notations, if $m=2$ and $f$ is even then the following equalities hold:
$$c_{11}^1=f/2-1,~c_{12}^1=c_{21}^1=c_{22}^1=f/2.$$
\end{lemm}

\begin{proof}
The statement of the lemma follows from~\cite[Eqs.~(18)-(19)]{Di} and Eq.~(9).
\end{proof}

Note that if $m=2$ and $f$ is even then $\mathcal{A}$ is the WL-closure of the Paley graph with $p$ vertices. The structure constants from Lemma~\ref{cyc1} can be computed from the parameters of the Paley graph using well-known formulas~\cite[Proposition~2.1.12]{CP}.

\begin{lemm}\label{cyc2}
In the above notations, let $m=3$. Put $x=9c_{23}^1-p-1$ and $y=c_{12}^1-c_{13}^1$. The following equalities hold:
$$4p=x^2+27y^2,~x\equiv 1 \mod~3;$$
$$9c_{11}^1=p-8+x;$$
$$18c_{12}^1=2p-4-x+9y,~18c_{13}^1=2p-4-x-9y.$$
\end{lemm}

\begin{proof}
The statement of the lemma  follows from~\cite[Eqs.~(33)-(35)]{Di} and Eq.~(9).
\end{proof}

\begin{rem1}
Note that the integers $x^2$ and $y^2$ can be uniquely determined by the equality $4p=x^2+27y^2$ (see~\cite[p.~7]{Di}).
\end{rem1}

The cyclotomic schemes with three nontrivial classes were studied in~\cite{Holl,Mathon}. In particular, intersection numbers of such schemes were expressed via some parameters in these papers. The next lemma provides a necessary and sufficient condition for the equality of some intersection numbers of a cyclotomic scheme with three nontrivial classes. This lemma will be used in the proof of Theorem~\ref{main1}. It is possible that the statement of this lemma is known but the authors did not find it and provide the direct proof to make the text self-contained.

\begin{lemm}\label{prime}
In the above notations, let $m=3$, $i\in\{1,2\}$, and $T_i=\{c_{11}^1+2(i-1),c_{11}^2,c_{11}^3\}$. Then $|T_i|$ is equal to~$2$ or~$3$. Moreover, $|T_i|=2$ if and only if $p=t^2+3i^2$ for some integer $t$.
\end{lemm}

\begin{proof}
Since $\frac{p-1}{3}$ is even, all basic sets of $\mathcal{A}$ are inverse-closed  sets. They are of the same size $\frac{p-1}{3}$. So Eq.~(2) implies that 
$$c_{11}^2=c_{12}^1~\text{and}~c_{11}^3=c_{13}^1.~\eqno(10)$$

The numbers $x$ and $y$ are defined as in Lemma~\ref{cyc2}. Assume that $c_{11}^2=c_{11}^3$. Then from the definition of $y$ and Eq.~(10) it follows that $y=0$. So $4p=x^2$ by the first equality from Lemma~\ref{cyc2}, a contradiction with the fact that $p$ is prime. Therefore $c_{11}^2\neq c_{11}^3$ and hence $|T_i|\geq 2$.

Suppose that $|T_i|=2$. Due to the above paragraph,  $c_{11}^1+2(i-1)=c_{11}^2$ or $c_{11}^1+2(i-1)=c_{11}^3$. The second and the third equalities from Lemma~\ref{cyc2} and Eq.~(8) imply that
$$x=4\pm 3y~\eqno(11)$$
if $i=1$, and
$$x=-8\pm 3y~\eqno(12)$$
if $i=2$. Together with the first equality from Lemma~\ref{cyc2}, this implies that
$$p=(3y\pm 1)^2+3$$
if $i=1$, and
$$p=(3y\pm 2)^2+12$$
if $i=2$.
Thus, $p=t^2+3$  for $t=3y\pm 1$ if $i=1$ and $p=t^2+12$  for $t=3y\pm 2$ if $i=2$.

Conversely, suppose that $p=t^2+3i^2$ for some integer $t$. Since $p-1$ is divisible by~$3$, we conclude that $t^2-i^2$ is divisible by~$3$. So exactly one of the numbers $t-i,t+i$ is divisible by~$3$. If $3$ divides $(t-i)$ then the integers $x_1=t+3i$ and $y_1=(t-i)/3$ satisfy the equality $4p=x_1^2+27y_1^2$. In view of Remark~1, we obtain $x=\pm x_1$ and $y=\pm y_1$. Note that $x=x_1=t+3$ if $i=1$ and $x=-x_1=-t-6$ if $i=2$ because $x\equiv 1 \mod~3$ by the first equality form Lemma~\ref{cyc2} and $t\equiv i\mod~3$. This implies that Eq.~(11) holds if $i=1$ and Eq.~(12) holds if $i=2$. If $3$ divides $(t+i)$ then the integers $x_1=t-3i$ and $y_1=(t+i)/3$ satisfy the equality $4p=x_1^2+27y_1^2$. Due to Remark~1, we have $x=\pm x_1$ and $y=\pm y_1$. Note that $x=-x_1=-t+3$ if $i=1$ and $x=x_1=t-6$ if $i=2$ because $x\equiv 1 \mod~3$ by the first equality from Lemma~\ref{cyc2} and $t\equiv (3-i)\mod~3$. Again, we obtain Eq.~(11) holds if $i=1$ and Eq.~(12) holds if $i=2$. Using Eqs.~(10)-(12), the second and the third equalities from Lemma~\ref{cyc2}, one can verify that $c_{11}^1+2(i-1)=c_{11}^2$ or $c_{11}^1+2(i-1)=c_{11}^3$. Thus, $|T_i|=2$. 
\end{proof}

\section{Proof of Theorems~\ref{main1} and~\ref{main2}}
Throughout this section, $G$ is a cyclic group of order $n$, $S\subseteq G$ with $e\notin S$ and $S=S^{-1}$, $\Gamma=\cay(G,S)$, and $\mathcal{A}=\WL(\Gamma)$ (see Section~5 for the notation). Suppose that $\rkwl(\Gamma)=4$. Due to Lemma~\ref{rankring}, we have $\rk(\mathcal{A})=4$. By the definition of an $S$-ring, there exist nonnegative integers $\alpha_1,\alpha_2,\alpha_3$ such that
$$\underline{S}^2=|S|X_0+\alpha_1\underline{X_1}+\alpha_2\underline{X_2}+\alpha_3\underline{X_3},$$
where $X_0=\{e\}$ and $\mathcal{S}(\mathcal{A})=\{X_0,X_1,X_2,X_3\}$. The set $S$ is a union of some $X_i$. Put $T(S)=\{\alpha_1,\alpha_2,\alpha_3\}$. Since $\rk(\mathcal{A})=4$, the graph $\Gamma$ is non-strongly regular and hence $|T(S)|\geq 2$. Lemma~\ref{deza} implies that $\Gamma$ is a Deza graph if and only if
$$|T(S)|=2.$$

\begin{lemm}\label{wl3}
In the above notations, suppose that $n$ is composite and $Y\neq \{e\}$ is an $\mathcal{A}$-set. Then $\WL(Y)=\mathcal{A}$ if and only if $Y\neq H^{\#}$ and $Y\neq G\setminus H$ for every $\mathcal{A}$-subgroup $H$.  
\end{lemm}

\begin{proof}
Put $\mathcal{B}=\WL(Y)$. The set $Y$ is an $\mathcal{A}$-set and $\mathcal{B}$ is the smallest $S$-ring for which $Y$ is a $\mathcal{B}$-set. So $\mathcal{B}\leq \mathcal{A}$ and $\rk(\mathcal{B})\leq \rk(\mathcal{A})=4$. Clearly, $\mathcal{B}=\mathcal{A}$ if and only if $\rk(\mathcal{B})=4$. Since $n$ is composite, Lemma~\ref{rank3} implies that $\rk(\mathcal{B})\leq 3$ if and only if $\mathcal{B}=\mathcal{T}_H\wr \mathcal{T}_{G/H}$ for some (possibly, trivial)  subgroup $H$ of $G$. In this case $Y=H^{\#}$ or $Y=G\setminus H$. Note that $H$ is an $\mathcal{A}$-subgroup because $H$ is a $\mathcal{B}$-subgroup. Thus, $\rk(\mathcal{B})=4$, i.e. $\mathcal{B}=\mathcal{A}$, if and only if  $Y\neq H^{\#}$ or $Y\neq G\setminus H$ for every $\mathcal{A}$-subgroup $H$. 
\end{proof}

\begin{prop}\label{proof1}
In the above notations, $\Gamma$ is a Deza graph if and only if $\Gamma$ is isomorphic to one of the graphs from Theorem~\ref{main1} and has parameters from Table~1.
\end{prop}

\begin{proof}
Since $\rk(\mathcal{A})=4$, we have $n=|G|\geq 4$. Suppose that $n=4$. Lemma~\ref{wl3} implies that $S\notin\{\{a^2\},\{a,a^3\},G^{\#}\}$, where $a$ is a generator of $G$. If $S=\{a^i\}$ or $S=\{a^2,a^i\}$, where $i\in\{1,3\}$, then $S\neq S^{-1}$. So Statement~6 of Theorem~\ref{rank4} does not hold for $\mathcal{A}$. Therefore, one of the Statements~1-5 of Theorem~\ref{rank4} holds for $\mathcal{A}$. We divide the proof into five cases depending on the statement of Theorem~\ref{rank4} which holds for $\mathcal{A}$.
\\
\\
\textbf{Case~1:} $\mathcal{A}=\mathcal{T}_L \otimes \mathcal{T}_U$.
In this case the basic sets of $\mathcal{A}$ are the following:
$$X_0=\{e\},~X_1=L^{\#},~X_2=U^{\#},~X_3=L^{\#}U^{\#}.$$
Put $|L|=k$ and $|U|=m$. Note that $\GCD(k,m)=1$ because $G$ is cyclic. Due to Lemma~\ref{wl3}, we conclude that $S=X_3$ or $S=X_1\cup X_2$. We claim that
$$\underline{S}^2=(k-1)(m-1)e+(k-2)(m-1)\underline{X_1}+(k-1)(m-2)\underline{X_2}+(k-2)(m-2)\underline{X_3}$$
in the former case and
$$\underline{S}^2=(k+m-2)e+(k-2)\underline{X_1}+(m-2)\underline{X_2}+2\underline{X_3}$$
in the latter case. 

The above equalities can be verified by the straightforward computation using Eq.~(3) in the same way. Let us verify, for the example, the latter equality. 
$$\underline{S}^2=(\underline{X_1}+\underline{X_2})^2=(\underline{L}^{\#}+\underline{U}^{\#})^2=
(|L|-1)e+(|L|-2)\underline{L}^{\#}+(|U|-1)e+(|U|-2)\underline{U}^{\#}+2\underline{L}^{\#}\underline{U}^{\#}$$
$$=(k+m-2)e+(k-2)\underline{X_1}+(m-2)\underline{X_2}+2\underline{X_3}.$$

It is easy to check that: (1) in the former case $|T(S)|=2$ if and only if $k=2$ or $m=2$; (2) in the latter case $|T(S)|=2$ if and only if $k=4$ or $m=4$. Together with Lemma~\ref{tensproduct} applied to $G_1=L$, $S_1=L^\#$, $\Gamma_1=\cay(L,L^\#)$, $G_2=U$, $S_2=U^\#$, $\Gamma_2=\cay(U,U^\#)$, this implies that $\Gamma$ is a Deza graph if and only if Statement~1 or Statement~2 of Theorem~\ref{main1} holds. Without loss of generality we may assume that $k=2$ if $S=X_3$ and $k=4$ if $S=X_1\cup X_2$. In the former case $\Gamma$ has parameters $(2m,m-1,m-2,0)$; in the latter case $\Gamma$ has parameters $(4m,m+2,m-2,2)$.
\\
\\
\textbf{Case~2:} $\mathcal{A}=(\mathcal{T}_L\wr \mathcal{T}_{U/L})\wr \mathcal{T}_{G/U}$.
In this case the basic sets of $\mathcal{A}$ are the following:
$$X_0=\{e\},~X_1=L^{\#},~X_2=U\setminus L,~X_3=G\setminus U.$$
Put $|L|=k$ and $|U|=f$. Lemma~\ref{wl3} yields that $S=X_2$ or $S=X_1\cup X_3$. The direct computation using Eq.~(3) shows that
$$\underline{S}^2=(f-k)e+(f-k)\underline{X_1}+(f-2k)\underline{X_2}+0\underline{X_3}$$
in the former case and
$$\underline{S}^2=(n-f+k-1)e+(n-f+k-2)\underline{X_1}+(n-f)\underline{X_2}+(n-2f+2k-2)\underline{X_3}$$
in the latter case. 

One can verify that: (1) in the former case $|T(S)|=2$ if and only if $f=2k$; (2) in the latter case $|T(S)|=2$ if and only if $k=2$. In the former case $|U:L|=2$ and the graph $\cay(U,S)=\cay(U,X_2)$ is a complete bipartite graph with parts $L$ and $U\setminus L$. Lemma~\ref{disjointunion} applied to $G=G$, $H=U$, and $S=X_2$ yields that $\Gamma$ is a Deza graph if and only if Statement~4 of Theorem~\ref{main1} holds for $l=\frac{n}{2k}$ and $m=k$. In this case $\Gamma$ has parameters $(n,k,k,0)=(2lm,m,m,0)$.

In the latter case from Lemma~\ref{disjointunion} applied to $G=U$, $H=L$, and $S=X_1$ it follows that $\cay(U,X_1)$ is a disjoint union of $m=\frac{f}{2}$  copies of $K_2$. Lemma~\ref{lexproduct} applied to $G=G$, $H=U$, $\Gamma_1=\cay(G/U, X_3/U)$, and $\Gamma_2=\cay(U,X_1)$ implies that $\Gamma$ is a Deza graph if and only if Statement~5 of Theorem~\ref{main1} holds for $m=\frac{f}{2}$ and $l=\frac{n}{f}$. In this case $\Gamma$ has parameters $(n,n-f+1,n-f,n-2f+2)=(2lm,2lm-2m+1,2lm-2m,2lm-4m+2)$.
\\
\\
\textbf{Case~3:} $\mathcal{A}=\cyc(\frac{p-1}{2},L)\wr \mathcal{T}_{G/L}$.
In this case the basic sets of $\mathcal{A}$ are the following:
$$X_0=\{e\},~X_1,~X_2,~X_3=G\setminus L,$$
where $X_1$ and $X_2$ are the orbits of the subgroup $K$ of $\aut(L)$ of index~$2$. Lemma~\ref{wl3} yields that $S\neq X_1 \cup X_2$ and $S\neq X_3$. So $S\in\{X_1, X_2, X_1\cup X_3, X_2\cup X_3\}$. One can see that there exists $\varphi\in \aut(G)$ such that $X_1^{\varphi}=X_2$ and $X_3^{\varphi}=X_3$ because $L$ is of prime order. So $\cay(G,X_1)\cong \cay(G,X_2)$ and $\cay(G,X_1\cup X_3)\cong \cay(G,X_2\cup X_3)$. Therefore we may assume that $S\in\{X_1,X_1\cup X_3\}$. Choosing an appropriate primitive root $l$ of $p$, we may assume that  $X_1=\{a^{l^{2j}}:j=0,\ldots,\frac{p-1}{2}-1\}$, where $a$ is a generator of $L$. Since $S=S^{-1}$, we conclude that $X_1=X_1^{-1}$ and hence $|K|=(p-1)/2$ is even. This means that $p\equiv 1\mod~4$.

The computation using Eq.~(3) and Lemma~\ref{cyc1} implies that
$$\underline{S}^2=\frac{p-1}{2}e+\frac{p-5}{4}\underline{X_1}+\frac{p-1}{4}\underline{X_2}+0\underline{X_3}$$
if $S=X_1$ and
$$\underline{S}^2=(n-\frac{p+1}{2})e+(n-\frac{3p+5}{4})\underline{X_1}+(n-\frac{3p+1}{4})\underline{X_2}+(n-p-1)\underline{X_3}$$
if $S=X_1\cup X_3$.  

It is easy to see that in the former case $|T(S)|=2$ if and only if $p=5$ and $S=\{x,x^{-1}\}$, where $x$ is a generator of $L$. Together with Lemma~\ref{disjointunion} applied to $G=G$, $H=L$, $S=X_1$, and $\Gamma=\cay(L,X_1)$, this implies that $\Gamma$ is a Deza graph if and only if Statement~3 of Theorem~\ref{main1} holds for $m=\frac{n}{5}$. In this case $\Gamma$ has parameters $(n,2,1,0)=(5m,2,1,0)$.

In the latter case one can verify that $|T(S)|=3$. So $\Gamma$ can not be a Deza graph in this case.
\\
\\
\textbf{Case~4:} $\mathcal{A}=\mathcal{T}_L\wr \cyc(\frac{p-1}{2},G/L)$.
In this case the basic sets of $\mathcal{A}$ are the following:
$$X_0=\{e\},~X_1,~X_2,~X_3=L^{\#},$$
where $L\leq \rad(X_1)\cap \rad(X_2)$ and $X_1/L$ and $X_2/L$ are the orbits of the subgroup $K$ of $\aut(G/L)$ of index~$2$. Put $|L|=m$. Note that $S\neq X_1 \cup X_2$ and $S\neq X_3$ by Lemma~\ref{wl3}. As in the previous case, $\cay(G,X_1)\cong \cay(G,X_2)$ and $\cay(G,X_1\cup X_2)\cong \cay(G,X_1\cup X_3)$ and hence we may assume that $S\in\{X_1,X_1\cup X_3\}$. Choosing an appropriate primitive root $l$ of $p$, we may assume that $X_1=\{a^{l^{2j}}:j=0,\ldots,\frac{p-1}{2}-1\}L$, where $a\in G\setminus L$. Since $S=S^{-1}$, we obtain $X_1=X_1^{-1}$ and hence $|K|=(p-1)/2$ is even. So $p\equiv 1\mod 4$.

The computation using Eq.~(3) and Lemma~\ref{cyc1} implies that
$$\underline{S}^2=\frac{(p-1)m}{2}e+\frac{(p-5)m}{4}\underline{X_1}+\frac{(p-1)m}{4}\underline{X_2}+\frac{(p-1)m}{2}\underline{X_3}$$
if $S=X_1$ and
$$\underline{S}^2=(m-1+\frac{(p-1)m}{2})e+(2m-2+\frac{(p-5)m}{4})\underline{X_1}+\frac{(p-1)m}{4}\underline{X_2}+(m-2+\frac{(p-1)m}{2})\underline{X_3}$$
if $S=X_1\cup X_3$. 

One can verify that $|T(S)|=3$ in the former case. So $\Gamma$ can not be a Deza graph in this case.

It can be checked that in the latter case $|T(S)|=2$ if and only if $m=2$. Note that $\cay(G/L,X_1/L)$ is the Paley graph with $p$ vertices. Therefore Lemma~\ref{lexproduct} applied to $G=G$, $H=L$, $\Gamma_1=\cay(G/L,X_1/L)$, and $\Gamma_2=\cay(L,X_3)$ implies that $\Gamma$ is a Deza graph if and only if Statement~6 of Theorem~\ref{main1} holds. In this case $\Gamma$ has parameters $(2p,p,p-1,\frac{p-1}{2})$.
\\
\\
\textbf{Case~5:} $G\cong C_p$ for a prime $p$ such that $p\equiv 1\mod~3$ and $\mathcal{A}=\cyc(\frac{p-1}{3},G)$. For every $i,j\in \{1,2,3\}$, there exists $\varphi\in \aut(G)$ such that $X_i^{\varphi}=X_j$ because $G$ is of prime order. So $\cay(G,X_i)\cong \cay(G,X_j)$ and $\cay(G,G^\#\setminus X_i)\cong \cay(G,G^\#\setminus X_j)$. Therefore we may assume that $S=X_1$ or $S=X_2\cup X_3$. Choosing an appropriate primitive root $l$ of $p$, we may assume that $X_i=\{a^{l^{3j+(i-1)}}:j=0,\ldots,\frac{p-1}{3}-1\}$, where $a$ is a generator of $G$ and $i\in\{1,2,3\}$. In the former case 
$$\underline{S}^2=\frac{p-1}{3}e+c_{11}^1\underline{X_1}+c_{11}^2\underline{X_2}+c_{11}^3\underline{X_3},$$
where $c_{ij}^k=c_{X_iX_j}^{X_k}$. In the latter case
$$\underline{S}^2=(\underline{G}-\underline{X_1}-e)^2=\frac{2(p-1)}{3}e+\alpha\underline{X_1}+\beta\underline{X_2}+\gamma\underline{X_3},$$
where
$$\alpha=(p-1)/3+1+c_{11}^1,~\beta=(p-1)/3-1+c_{11}^2,~\text{and}~\gamma=(p-1)/3-1+c_{11}^3.$$
In the former case $T(S)=T_1$ and in the latter case $T(S)=T_2+(p-1)/3-1=\{t+(p-1)/3-1:~t\in T_2\}$, where $T_1$ and $T_2$ are defined as in Lemma~\ref{prime}. From Lemma~\ref{prime} it follows that $|T(S)|=2$ or, equivalently, $\Gamma$ is a Deza graph  if and only if Statement~7 or~8 of Theorem~\ref{main1} holds. 
\end{proof}

Theorem~\ref{main1} immediately follows from Proposition~\ref{proof1}.

\begin{prop}\label{proof2}
In the above notations, $\dimwl(\Gamma)=2$ whenever one of the Statements~1-5 of Theorem~\ref{main1} holds for $\Gamma$ and $\dimwl(\Gamma)\in\{2,3\}$ whenever one of the Statements~6-8 of Theorem~\ref{main1} holds for $\Gamma$.
\end{prop}

\begin{proof}
The graph $\Gamma$ is regular because $\Gamma$ is a Cayley graph. Since $\rkwl(\Gamma)=4$, the graph $\Gamma$ is non-strongly regular. So Lemma~\ref{dim1} yields that
$$\dimwl(\Gamma)\geq 2.~\eqno(13)$$

Suppose that one of the Statements~1-2 or~4-5 of Theorem~\ref{main1} holds for $\Gamma$. Then one of the Statements~1-2 of Theorem~\ref{rank4} holds for $\mathcal{A}$ (Cases~1-2 in the proof of Proposition~\ref{proof1}). The $S$-ring $\mathcal{A}$ is separable by Lemma~\ref{separ}. So $\dimwl(\Gamma)\leq 2$ by Statement~1 of Lemma~\ref{wldim}. Together with Eq.~(13), this implies that $\dimwl(\Gamma)=2$.

Suppose that Statement~3 of Theorem~\ref{main1} holds for $\Gamma$. Then $\mathcal{A}=\cyc(\frac{p-1}{2},L)\wr \mathcal{T}_{G/L}$, where $|L|=p=5$ (Case~3 in the proof of Proposition~\ref{proof1}). The $S$-ring $\mathcal{A}_L$ is separable because $p=5$ and every $S$-ring over a group of order at most~$14$ separable (see~\cite[p.~64]{CP}). So $\mathcal{A}$ is separable by Lemma~\ref{separ}. Therefore, $\dimwl(\Gamma)\leq 2$ by Statement~1 of Lemma~\ref{wldim}. In view of Eq.~(13), we obtain $\dimwl(\Gamma)=2$.

If Statement~6 of Theorem~\ref{main1} holds for $\Gamma$ then $\mathcal{A}=\mathcal{T}_L\wr \cyc(\frac{p-1}{2},G/L)$, where $|L|=2$ (Case~4 in the proof of Proposition~\ref{proof1}). If one of the Statements~7-8 of Theorem~\ref{main1} holds for $\Gamma$ then $\mathcal{A}=\cyc(\frac{p-1}{3},G)$ (Case~5 in the proof of Proposition~\ref{proof1}). We have $\dimwl(\Gamma)\leq 3$ by Lemma~\ref{paley} in the former case and by Lemma~\ref{onepointnormal} in the latter case. Together with Eq.~(13), this implies that $\dimwl(\Gamma)\in \{2,3\}$.
\end{proof}

It seems that $\dimwl(\Gamma)=3$ if one of the Statements~6-8 of Theorem~\ref{main1} holds for $\Gamma$. However, checking this looks complicated.

\begin{proof}[Proof of Theorem~\ref{main2}]
Let $\mathcal{K}$ be the class of circulant Deza graphs of WL-rank~$4$. From Proposition~\ref{proof2} it follows that $\dimwl(\mathcal{K})\leq 3$. 

Suppose that Statement~6 of Theorem~\ref{main1} holds for $\Gamma$. Then $\mathcal{A}=\mathcal{T}_L\wr \cyc(\frac{p-1}{2},G/L)$, where $|L|=2$ (Case~4 in the proof of Proposition~\ref{proof1}). Assume that $\dimwl(\Gamma)\leq 2$. Then $\mathcal{A}$ is separable by Statement~1 of Lemma~\ref{wldim}. So $\cyc(\frac{p-1}{2},G/L)$ is separable by Lemma~\ref{separ}. However, if $p=29$ then according to the list of association schemes~\cite{HM}, there exists an association scheme  which is algebraically isomorphic but not isomorphic to $\mathcal{X}(\cyc(\frac{p-1}{2},G/L))$. Therefore $\cyc(\frac{p-1}{2},G/L)$ is nonseparable  and hence $\dimwl(\Gamma)=3$ whenever $p=29$. Thus, $\dimwl(\mathcal{K})=3$.
\end{proof}

\section{Strictly Deza circulant graphs of WL-rank $>4$}
In this section we study the WL-rank and the WL-dimension of some strictly Deza circulant graphs whose WL-rank is greater than~$4$. 

\subsection{Family~1}
The graphs described below first appeared in~\cite{GGSh}. Let $p$ and $q$ be distinct odd primes with $q-p=4$. Suppose that $G=P\times Q$, where $P\cong C_p$ and $Q\cong C_q$. Let $M$ and $N$ be subgroups of index~$2$ in $\aut(P)$ and $\aut(Q)$ respectively. Denote the canonical epimorphisms from $\aut(P)$ to $\aut(P)/M$ and from $\aut(Q)$ to $\aut(Q)/N$ by $\pi_1$ and $\pi_2$ respectively. There exists the unique isomorphism $\psi$ from $\aut(P)/M$ to $\aut(Q)/N$ because $\aut(P)/M\cong \aut(Q)/N\cong C_2$. One can form the subdirect product $K=K(p,q)$ of $\aut(P)$ and $\aut(Q)$ in the following way:
$$K=K(p,q)=\{(\sigma_1,\sigma_2)\in \aut(P)\times \aut(Q):~(\sigma_1^{\pi_1})^{\psi}=\sigma_2^{\pi_2}\}.~\eqno(14)$$
Observe that $K\geq M\times N$ and $|(\aut(P)\times \aut(Q)):K|=|K:(M\times N)|=2$. Let $P_1,P_2$ and $Q_1,Q_2$ be the nontrivial orbits of $M$ on $P$ and $N$ on $Q$ respectively. It is easy to verify that the orbits of $M\times N$ on $G$ are the following:
$$\{e\}, P_1, P_2, Q_1, Q_2, P_1Q_1, P_1Q_2, P_2Q_1, P_2Q_2.$$

Put $\mathcal{A}=\cyc(K,G)$. One can see that the basic sets of $\mathcal{A}$, i.e. the orbits of $K$, are the following:
$$X_0=\{e\},~X_1=P^{\#},~X_2=Q^{\#},~X_3=P_1Q_1\cup P_2Q_2,~X_4=P_1Q_2\cup P_2Q_1.$$ 
Lemma~\ref{circ} implies that $\mathcal{A}$ is a normal cyclotomic $S$-ring with trivial radical. Put $S=X_2\cup X_4$ and $\Gamma=\cay(G,S)$. Note that $S=S^{-1}$. The set $S$ coincides with the set $S_0\cup S_1\cup S_2$ from~\cite[p.~20]{GGSh}. So $\Gamma$ is indeed isomorphic to the graph constructed in~\cite[p.~20]{GGSh}.

\begin{prop}\label{f1}
The graph $\Gamma$ is a strictly Deza graph with parameters from the first line of Table~2, $\WL(\Gamma)=\mathcal{A}$, $\rkwl(\Gamma)=5$, and $\dimwl(\Gamma)\in \{2,3\}$.
\end{prop}

\begin{proof}
The graph $\Gamma$ is a strictly Deza graph with the above parameters by~\cite{GGSh}. Put $\mathcal{B}=\WL(\Gamma)$. Let us prove that $\mathcal{B}=\mathcal{A}$. Observe that $\mathcal{B}\leq \mathcal{A}$ because $\underline{S}\in \mathcal{A}$. One of the Statements~2-4 of Lemma~\ref{circ} holds for $\mathcal{B}$. The set $S$ contains generators of $G$ and each subset of $S$  has trivial radical. So $\mathcal{B}$ is not a generalized wreath product of two $S$-rings. Therefore Statement~2 or Statement~4 of Lemma~\ref{circ} holds for $\mathcal{B}$. In both cases $\underline{X_1},\underline{X_2}\in \mathcal{B}$. Since $\underline{S}\in \mathcal{B}$ and $X_4=S\setminus X_2$, we obtain $\underline{X_4}\in \mathcal{B}$. Clearly, $X_3=G\setminus (X_0\cup X_1\cup X_2\cup X_4)$. Therefore, $\underline{X_3}\in \mathcal{B}$. This yields that $\mathcal{B}\geq \mathcal{A}$. Thus, $\mathcal{B}=\mathcal{A}$. Obviously, $\rkwl(\Gamma)=\rk(\mathcal{A})=5$.

The graph $\Gamma$ is non-strongly regular because $\rkwl(\Gamma)=5$. So $\dimwl(\Gamma)\geq 2$ by Lemma~\ref{dim1}. Since $\mathcal{A}$ is a normal cyclotomic $S$-ring with trivial radical, $\dimwl(\Gamma)\leq 3$ by Lemma~\ref{onepointnormal} and hence $\dimwl(\Gamma)\in \{2,3\}$. 
\end{proof}

Note that it is not known whether Family~$1$ is infinite. 

\subsection{Family~2}
The graphs from this subsection were introduced in~\cite[p.~7]{HKM} as Deza graphs whose adjacency matrices can be obtained from a regular graphical Hadamard $4\times 4$-matrix. Let $G=A\times B$, where $A\cong C_4$ and $B$ is a cyclic group of odd order~$k\geq 3$. Let $a$ be a generator of $A$. Put
$$X_0=\{e\},~X_1=\{a^2\},~X_2=\{a,a^3\},~X_3=B^{\#},~X_4=a^2B^{\#},~X_5=\{a,a^3\}B^{\#}.$$
The partition $\mathcal{S}=\{X_0,X_1,X_2,X_3,X_4,X_5\}$ of $G$ defines the $S$-ring $\mathcal{A}$ over $G$ such that $\mathcal{A}=\mathcal{A}_A\otimes \mathcal{T}_B$ and $\mathcal{A}_A\cong \mathbb{Z}C_2\wr\mathbb{Z}C_2$. Put $S=X_1\cup X_3\cup X_5$ and $\Gamma=\cay(G,S)$. Note that $S=S^{-1}$. 

Let $\sigma=(23)\in \sym(M)$, where $M=\{0,1,2,3\}$. If we number the elements of $G$ so that the elements from $Ba^i$ have the numbers from $1+i^{\sigma}k$ to $(1+i^{\sigma})k$ then the adjacency matrix of $\Gamma$ is of the following form
\begin{equation*}
\begin{pmatrix}
J_k-I_k & J_k-I_k & J_k-I_k & I_k \\
J_k-I_k & J_k-I_k & I_k & J_k-I_k \\
J_k-I_k  & I_k & J_k-I_k & J_k-I_k  \\
I_k & J_k-I_k & J_k-I_k & J_k-I_k
\end{pmatrix},
\end{equation*}
where $I_k$ and $J_k$ are the identity and all-one matrices of size $k\times k$ respectively. The above matrix coincides with the matrix constructed in~\cite[Construction~4.8]{HKM}. So $\Gamma$ is indeed isomorphic to the graph from~\cite[p.~7]{HKM}.

\begin{prop}\label{f2}
The graph $\Gamma$ is a strictly Deza graph with parameters from the second line of Table~2, $\WL(\Gamma)=\mathcal{A}$, $\rkwl(\Gamma)=6$, and $\dimwl(\Gamma)=2$.
\end{prop}

\begin{proof}
The graph $\Gamma$ is a strictly Deza graph with the above parameters by~\cite[p.~7]{HKM}. Put $\mathcal{B}=\WL(\Gamma)$. Let us prove that $\mathcal{B}=\mathcal{A}$. Clearly, $\mathcal{B}\leq \mathcal{A}$ because $\underline{S}\in \mathcal{A}$. The direct computation using Eq.~(2) shows that
$$\underline{S}^2=(3k-2)e+2(k-1)(\underline{X_1}+\underline{X_2}+\underline{X_4}+\underline{X_5})+3(k-2)\underline{X_3}.$$
From the above equality it follows that $\underline{X_3}\in \mathcal{B}$ because only $\underline{X_3}$ enters $\underline{S}^2$ with coefficient $3(k-2)$. So $\underline{S_1}\in \mathcal{B}$, where $S_1=S\setminus X_3$. The straightforward check yields that
$$\underline{S_1}^2=(2k-1)e+2(k-1)\underline{X_1}+0\underline{X_2}+2(k-2)\underline{X_3}+2(k-2)\underline{X_4}+2\underline{X_5}.$$
Due to the above equality, the elements $\underline{X_1},\underline{X_2},\underline{X_3}+\underline{X_4},\underline{X_5}$ enters $\underline{S_1}^2$ with pairwise distinct coefficients. So we obtain $\underline{X_1},\underline{X_2},\underline{X_3}+\underline{X_4},\underline{X_5}\in \mathcal{B}$. Since also $\underline{X_3}\in \mathcal{B}$, we have $\underline{X_4}\in \mathcal{B}$. Thus, $\mathcal{B}\geq \mathcal{A}$ and hence $\mathcal{B}=\mathcal{A}$. It is easy to see that $\rkwl(\Gamma)=\rk(\mathcal{A})=6$.

The graph $\Gamma$ is non-strongly regular because $\rkwl(\Gamma)=5$. So $\dimwl(\Gamma)\geq 2$ by Lemma~\ref{dim1}. The $S$-ring $\mathcal{A}$ is separable by Lemma~\ref{separ} and hence $\dimwl(\Gamma)\leq 2$ by Statement~1 of Lemma~\ref{wldim}. Thus, $\dimwl(\Gamma)=2$.
\end{proof}

\subsection{Sporadic graphs}
Two graphs from this subsection were found by computer calculations and described in~\cite{GGSh}. Let $G\cong C_n$, where $n\in\{8,9\}$, and $a$ a generator of $G$. Let $\sigma\in \aut(G)$ such that $a^{\sigma}=a^{-1}$ and $K=\langle \sigma\rangle$. Put $\mathcal{A}=\cyc(K,G)$. Lemma~\ref{circ} implies that $\mathcal{A}$ is a normal cyclotomic $S$-ring with trivial radical. If $n=8$ then put $S=\{a,a^7,a^2,a^6\}$; if $n=9$ then put $S=\{a,a^8,a^2,a^7\}$. Put also $\Gamma=\cay(G,S)$.

\begin{prop}\label{sp}
The graph $\Gamma$ is a strictly Deza graph with parameters from the third line if $n=8$ and the fourth line if $n=9$ of Table~2, $\WL(\Gamma)=\mathcal{A}$, $\rkwl(\Gamma)=5$, and $\dimwl(\Gamma)=2$.
\end{prop}

\begin{proof}
All statements of the lemma except the last one can be easily verified by computer calculations. Lemma~\ref{dim1} yields that $\dimwl(\Gamma)\geq 2$. From~\cite[p.~64]{CP} it follows that every $S$-ring over a group of order at most~$14$ is separable. So $\dimwl(\Gamma)\leq 2$ by Statement~1 of Lemma~\ref{wldim}. Thus, $\dimwl(\Gamma)=2$.
\end{proof}

\subsection{Proof of Theorem~\ref{main3}}

The analysis of the computational results presented in~\cite{GGSh} implies that every strictly Deza circulant graph with at most~$95$ vertices belongs to one of the families from Theorem~\ref{main1}, or Family~$1$, or Family~$2$, or is isomorphic to one of two sporadic graphs. Therefore the statement of Theorem~\ref{main3} follows from Theorem~\ref{main2}, Proposition~\ref{f1}, Proposition~\ref{f2}, and Proposition~\ref{sp}.

\appendix
\setcounter{secnumdepth}{0}
\section{Appendix}
We collect an information on the circulant Deza graphs which occur in the paper in two tables below. Table~$1$ contains an information on the graphs from Theorem~\ref{main1}. The number of a graph in the first column corresponds to the number of statement of Theorem~\ref{main1} that holds for the graph. In the second column we list parameters of Deza graphs and in the third column we indicate which graphs are strictly Deza graphs (SDG).

In the fourth column we provide the description of the automorphism groups of graphs from Theorem~\ref{main1}. This description can be deduced from Eqs.~(1), (4), (5), and Lemma~\ref{autnorm}. The parameters $a$ and $b$ of $\Gamma_7$ and $\Gamma_8$ can be expressed via cyclotomic numbers (see Case~$5$ in the proof of Proposition~\ref{proof1}). In the fifth column we provide an information on WL-dimension of the graphs from Theorem~\ref{main1}. This information is taken from Proposition~\ref{proof2}.

In the sixth column we indicate which graphs from Theorem~\ref{main1} are divisible design. A $k$-regular graph with $n$ vertices is called a \emph{divisible design graph} (\emph{DDG}) with parameters $(n,k,a,b,m,l)$ if its vertex set can be partitioned into $m$ classes of size $l$, such that every two distinct vertices from the same class have $a$ common neighbors and every two vertices from different classes have $b$ common neighbors. The notion of a divisible design graph was introduced in~\cite{HKM}. One of the reasons to study divisible design graphs, in particular divisible design Deza graphs, is that they make a link between combinatorial design theory and algebraic graph theory. For more details on divisible design graphs, we refer the readers to~\cite{HKM,KS}. From~\cite[Theorem~1.1]{KS} it follows that a Cayley Deza graph $\cay(G,S)$ over a group $G$ is divisible design if and only if $A\cup\{e\}$ or $B\cup\{e\}$, where $A$ and $B$ are defined as in Lemma~\ref{deza}, is a subgroup of $G$. Using this statement, one can easily check which of the considered graphs are divisible design.

\begin{table}[h]
{\small
\scalebox{0.87}{\begin{tabular}{|l|l|l|l|l|l|}
  \hline
$\Gamma$ & parameters of $\Gamma$ & SDG & $\aut(\Gamma)$ & WL-dim & DDG \\
  \hline
  
$\Gamma_1$ & $(4m,m+2,m-2,2)$, $m>1$ is odd & yes & $\sym(4)\times \sym(m)$ & $2$ &yes \\ \hline

$\Gamma_2$ & $(2m,m-1,m-2,0)$, $m>1$ is odd & no & $C_2\times \sym(m)$ & $2$ &yes  \\ \hline
	
$\Gamma_3$ & $(5m,2,1,0)$, $m>1$ & no & $(C_5\rtimes C_2)\wr \sym(m)$ & $2$ &no  \\ \hline

$\Gamma_4$ & $(2lm,m,m,0)$, $l,m>1$ & no & $(\sym(m)\wr C_2)\wr \sym(l)$ & $2$ &yes  \\ \hline

$\Gamma_5$ & $(2lm,2lm-2m+1,2lm-2m,2lm-4m+2)$, $l,m>1$  & yes & $(C_2\wr\sym(m))\wr \sym(l)$ & $2$ &yes  \\ \hline

$\Gamma_6$ & $(2p,p,p-1,\frac{p-1}{2})$, $p\equiv 1\mod~4$ is prime & yes & $C_2\wr(C_p\rtimes C_{\frac{p-1}{2}})$ & $2$ or $3$ &yes \\ \hline

$\Gamma_7$ & $(p,\frac{p-1}{3},b,a)$, $p=t^2+3$ is prime & yes & $C_p\rtimes C_{\frac{p-1}{3}}$ & $2$ or $3$ &no \\ \hline

$\Gamma_8$ & $(p,2\frac{p-1}{3},b,a)$, $p=t^2+12$ is prime & yes & $C_p\rtimes C_{\frac{p-1}{3}}$ & $2$ or $3$ &no \\ \hline

\end{tabular}
}}
\vspace{\baselineskip}
\caption{Deza circulant graphs of WL-rank~$4$.}
\end{table}

Table~$2$ contains an information on graphs from Section~$8$. The first and the second lines are concerned with Families~$1$ and~$2$ respectively; the third and the fourth lines concerned with sporadic graphs. As in the previous table, the information on the automorphism groups can be deduced from  Eqs.~(1), (3), (4), and Lemma~\ref{autnorm}. The group $K(p,q)$ in the first line and the third column is defined by Eq.~(14). The information on WL-dimension from the fifth column is taken from Propositions~\ref{f1},~\ref{f2}, and~\ref{sp}. One can check whether $\Gamma$ is divisible design using Lemma~\ref{deza} and~\cite[Theorem~1.1]{KS}.

\begin{table}[h]
{\small
\begin{tabular}{|l|l|l|l|l|}
  \hline
 $\Gamma$ & parameters of $\Gamma$ & $\aut(\Gamma)$ & WL-dim  & DDG  \\
  \hline
  
$\Gamma_1$ & \makecell{$(pq,(pq+3)/2,(pq+7)/4, (pq+3)/4)$, \\$p$ and $q$ are distinct primes,\\$q-p=4$, $p\equiv q\equiv 3\mod~4$}  & $C_{pq}\rtimes K(p,q)$ & $2$ or $3$ & no\\ \hline
	
$\Gamma_2$ & $(4m,3m-2,3(m-2),2(m-1))$, $m>1$ is odd & $(C_2\wr C_2)\times \sym(m)$& $2$ &yes \\ \hline

$\Gamma_3$ & $(8,4,2,1)$ &  $C_8\rtimes C_2$& $2$ &no \\ \hline

$\Gamma_4$ & $(9,4,2,1)$ &  $C_9\rtimes C_2$& $2$ &no \\ \hline

\end{tabular}
}
\vspace{\baselineskip}
\caption{Strictly Deza circulant graphs of WL-rank $>4$.}

\end{table}

\end{document}